\documentclass{amsart}
\usepackage{amssymb, enumerate}

\usepackage[dvipsnames]{xcolor}

\usepackage{amssymb}
\usepackage{amsthm}
\usepackage{amsmath}
\usepackage{paralist}
\usepackage{graphics}
\usepackage{epsfig}
\usepackage[colorlinks=true]{hyperref}
\hypersetup{urlcolor=blue, citecolor=red}

\usepackage{graphicx}

\newcommand{\dvol}{\,d\operatorname{vol}}

\newtheorem{theorem}{Theorem}[section]
\newtheorem{lemma}[theorem]{Lemma}
\newtheorem{corollary}[theorem]{Corollary}
\newtheorem{proposition}[theorem]{Proposition}

\theoremstyle{definition}
\newtheorem{definition}[theorem]{Definition}
\newtheorem{example}[theorem]{Example}
\newtheorem{notation}[theorem]{Notation}

\theoremstyle{remark}
\newtheorem{remark}[theorem]{Remark}

\usepackage[normalem]{ulem}
\renewcommand\sout{\bgroup\markoverwith
{\textcolor{red}{\rule[0.7ex]{3pt}{1.4pt}}}\ULon}

\newcommand\seq{\, = \,}
\newcommand\define{\mathrel{\ := \ }}
\newcommand\ede{\define}

\newcommand\esssup{\operatorname{ess-sup}}
\newcommand\dist{\operatorname{dist}}
\newcommand\Hom{\operatorname{Hom}}
\newcommand\End{\operatorname{End}}

\newcommand\Def{\operatorname{Def}\,}

\newcommand{\rinj}{\mathop{r_{\mathrm{inj}}}}

\newcommand{\CC}{\mathbb C}

\newcommand{\NN}{\mathbb N}

\newcommand{\RR}{\mathbb R}

\newcommand{\ZZ}{\mathbb Z}

\newcommand{\maD}{\mathcal D}

\newcommand{\maO}{\mathcal O}

\newcommand{\CI}{{\mathcal C}^{\infty}}
\newcommand{\CIc}{{\mathcal C}^{\infty}_{\text{c}}}

\newcommand\pa{\partial}

\newcommand\dimM{{n}}

\newcommand{\tps}[2]{\texorpdfstring{$#1$}{#2}}

\numberwithin{equation}{section}

\begin{document}

\title[Bounded geometry and the deformation operator]{The
$L^2$-unique continuation property on manifolds with bounded
geometry and the deformation operator}


\author[N. Gro{\ss}e]{Nadine Gro{\ss}e} \address{Mathematisches
  Institut, Universit\"at Freiburg, 79104 Freiburg, Germany}
\email{nadine.grosse@math.uni-freiburg.de}

\author[M. Kohr]{Mirela Kohr}
\address{Faculty of Mathematics and Computer Science,
Babe\c{s}-Bolyai University, 1 M. Kog\u{a}l\-niceanu Str., 400084
Cluj-Napoca, Romania} \email{mkohr@math.ubbcluj.ro}

\author[V. Nistor]{Victor Nistor} \address{Universit\'{e} de Lorraine,
  UFR MIM, Ile du Saulcy, CS 50128, 57045 METZ, France and
  Inst. Math. Romanian Acad.  PO BOX 1-764, 014700 Bucharest Romania}
\email{victor.nistor@univ-lorraine.fr}

\thanks{M.K. has been partially supported by
  the Babe\c{s}-Bolyai University Grant SEED/2022. N.G. has been partially supported by SPP2026 (DFG)}

\date\today

\subjclass[2000]{Primary 35R01; Secondary 35B65, 76M99.}

\date{\today}

\keywords{
Sobolev spaces, elliptic and uniformly strongly elliptic operators, invertibility,
Hadamard well-posedness, regularity estimates, unique continuation, Fredholm operators.
}
\dedicatory{In memory of Professor Gabriela Kohr, with deep respect}

\begin{abstract}
  A differential operator $T$ satisfies the $L^2$-unique continuation property if
  every $L^2$-solution of $T$ that vanishes on an open subset vanishes
  identically. We study the $L^2$-unique continuation property of an
  operator $T$ acting on a manifold with bounded geometry. In particular,
  we establish some connections between this property and the regularity
  properties of $T$. As an application, we prove that the deformation
  operator on a manifold with bounded geometry satisfies regularity and
  $L^2$-unique continuation properties. As another application, we prove that
  suitable elliptic operators are invertible (Hadamard well-posedness).
  Our results apply to compact manifolds, which have bounded geometry.
\end{abstract}

\maketitle
\tableofcontents

\section{Introduction}
\label{sec1}

Let $M$ be a Riemannian manifold and $E, F \to M$ be Hermitian vector bundles
and $T$ be differential operator acting from sections of $E$ to sections of $F$.
We shall study the following \emph{weak form} of the unique continuation property
(see below for references to the usual unique continuation property):

\begin{definition} \label{def-L2-cont}
  We shall say that $T$ has the \textit{ $L^2$-unique continuation property} if, 
  given $u \in L^2(M; E)$ that vanishes in an open subset $U \subset M$ and 
  satisfies $T  u \seq 0$, then $u \seq 0$  \emph{everywhere} on $M$.
\end{definition}

In our results we shall assume that the manifold $M$ has \emph{bounded geometry}
(Definition \ref{def_bdd_geo_no_bdy}). 
The main results of this paper relate are:
\begin{itemize}
  \item Results relating the $L^2$-unique continuation property of $T$ 
  to the regularity properties of the operator $T$.

  \item A proof that the deformation operator $\Def$ and the 
  deformation Laplacian $\mathbf{L} := 2\operatorname{Def}^*\Def$ have
  the $L^2$-unique continuation property.
  
  \item A proof that certain PDEs (including some of the form $\mathbf{L} + V$) 
  are well-posed in the sense of Hadamard. 
\end{itemize}

Let us first discuss \emph{the connections between the $L^2$-unique continuation property
and regularity.} The result that we need and prove is the following extension of a regularity result 
from \cite{GN17}, which in turn extends the classical regularity result on compact manifolds
(our convention is that $\NN = \{1, 2, \ldots\}$ and $\ZZ_+ = \NN \cup \{0\}$):

\begin{theorem}\label{thm.reg.D}
  Let $s, \rho \in \ZZ_+$, $s \le \rho$, and $t \in \RR$.
  Let $M$ be a manifold with bounded geometry and $E, F \to M$ be two Hermitian
  vector bundles with totally bounded curvature (Definition~\ref{def_ttly_bdd_curv}).
  Let $P$ be a second order, uniformly strongly elliptic differential operator with coefficients
  in $W^{\rho+3, \infty}$ acting on sections of $E$ to sections of $F$.
  Then there exists $c_{s,t} > 0$ such that, if $u \in H^{t}(M; E)$ and $Pu \in H^{s}(M; F)$, then
  $u \in H^{s + 2}(M; E)$ and
  \begin{equation*}
    \| u \|_{H^{s + 2}(M; E)} \leq c_{s,t}
    \left( \|P u\|_{H^{s}(M; F)} + \| u \|_{H^{t}(M; E)} \right).
  \end{equation*}
\end{theorem}

Related results for \emph{scalar differential operators with smooth
coefficients} were obtained in \cite{kordyukovLp1, kordyukovLp2} using the
calculus of uniform pseudodifferential operators on manifolds with bounded geometry.
(More regularity results will be included in \cite{GKN-23}, so we are not striving in this 
paper for the greatest generality.)
The proof of Theorem~\ref{thm.reg.D} is based on the regularity result in \cite{GN17} and
a generalization of the classical Rellich-Kondrachov Theorem, which, in the case of manifolds
with bounded geometry, is an embedding result, see Proposition~\ref{prop.l.reg}. Our proof
uses a description of Sobolev spaces using partitions of unity \cite{GrosseSchneider}, which
we also recall. The regularity result allows us to prove that certain differential
operators are closed or that they are self-adjoint.

The main application in this paper is to use Theorem~\ref{thm.reg.D} to
prove that the \emph{deformation operator} on a manifold with
bounded geometry satisfies \emph{regularity and $L^2$-unique
continuation properties.} Another application is to prove that certain 
elliptic operators are invertible (i.e. are well-posed in the sense of
Hadamard). In turn, these results on the deformation operator
are useful for the study of the Stokes system on manifolds with bounded 
geometry \cite{KohrNistorWendland1}.

As we have alluded to above, this paper uses a weaker version of the
unique continuation property than the usual one. Thus, in fact, this
paper is more about regularity, self-adjointness, and well-posedness of
elliptic operators than about their unique continuation property. The unique 
continuation property (in its strong form) is a \emph{very important topic} 
in mathematics. Its origins go back at least to classical results of 
Holmgren and Carleman at the very beginning of the twentieth century. 
It is hard to give a complete list of references. See 
however the following very incomplete lists
\cite{Hormander1, LebeauBook, LernerBook, SoggeFourier, TataruNotes, Taylor1, TrevesBookPDE} 
for some general references as well as the following research 
papers \cite{KenigSogge2, Choulli, Salo, Garofalo, KenigSogge1, KochTataru, LebeauCarleman,
	LernerUC, TataruUC} and the references therein for more information.

The paper is organized as follows. In Section~\ref{sec2}, we set up
notation and we recall the definition of the Sobolev spaces and of
the differential operators that we will use. That sections contains only
standard material. In Section~\ref{sec3}, we recall the definition
of manifolds with bounded geometry as well as the equivalent
definition of Sobolev spaces on such manifolds using suitable
partitions of unity. We also prove our Rellich-Kondrachov embedding
result, Proposition~\ref{prop.l.reg}.
Section~\ref{sec4} is devoted to Theorem~\ref{thm.reg.D}
and some direct applications. Section~\ref{sec5} contains our application,
which is that the deformation operator $\Def$ and the deformation Laplacian
$\mathbf{L}\define 2\operatorname{Def}^*\Def$ have the $L^2$-unique
continuation property. We also include there an invertibility result
for certain classes of operators that include certain operators of the 
form $\mathbf{L} + V$.

\section{Background material and notation}
\label{sec2}

We begin with some background material, for the benefit of the reader.
More precisely, we introduce here the needed function spaces and
their associated differential operators with the definitions used
in this paper. Since we are interested in
systems, all our functions spaces and differential operators will be
defined using sections of vector bundles. We also use this opportunity to
fix the notation. This section contains no new results.
See \cite{GN17} and \cite{KohrNistor1} for more details.

\emph{Throughout this paper, $M$ will be a (usually non-compact) connected
smooth Riemannian manifold without boundary with metric $g$.}
We use the convention that ${\mathbb Z}_+ = \{0, 1, 2, \ldots, \}$ and
${\mathbb N}={\mathbb Z}_+ \smallsetminus \{0\}$.

\subsection{Vector bundles, connections, and tensor products}
Let $E \to M$ denote a smooth (real or complex) vector bundle endowed with a
fiberwise metric $(\, ,\, )_E$ over our Riemannian manifold $M$ and let $\CI(M; E)$
be the set of \emph{smooth} sections of $E$. Unless otherwise stated, we assume
that $E$ is endowed with a metric preserving connection
\begin{equation*}
  \nabla^E \colon \CI(M; E) \to \CI(M; E \otimes T^{*}M)\,.
\end{equation*}
Let $F \to M$ be another vector bundle with connection $\nabla^F$. Then
$E \otimes F$ will be endowed with the product connection $\nabla^{E
  \otimes F}(\xi \otimes \eta) \define (\nabla^{E} \xi) \otimes \eta +
\xi \otimes \nabla^{F}\eta$. We will use the notation
$V^{\otimes k} \define V \otimes V \otimes \ldots \otimes V$ for the tensor product
of $V$ with itself of $k$-times. In particular, all the tensor product
bundles $E \otimes T^{*\otimes k}M \define E \otimes (T^*M)^{\otimes k}$ will
be endowed with the induced tensor product connections  from $E$
and from the Levi-Civita connection on $TM$. (Recall that the Levi-Civita
connection on $TM$ is the unique metric preserving torsion free connection
on $TM$ \cite{Petersen}.)

When there is no risk of confusion, we shall drop the notation for the vector
bundle $E$ from the notation of its connection $\nabla^E$ and write simply $\nabla$
instead.

\subsection{Covariant Sobolev spaces}
\label{ssec.covariant}
Let us recall now the definition of the covariant Sobolev spaces (or
$\nabla$-Sobolev spaces) on Riemannian manifolds. These spaces are defined using
connections, see, for instance \cite{AGN1, HebeyBook, HebeyRobert, KohrNistor1}.

\subsubsection{Positive integer order covariant Sobolev spaces}
Let $p\in [1,\infty]$ and $E \to M$ be a given Hermitian vector bundle with
$\|\cdot \|_E$ denoting the fiberwise norm on $E$.
Let $\dvol_g$ denote the induced volume form on $M$ (defined using
the given metric $g$ on $M$).
We then have the usual formula for the $L^p$-norm:
\begin{equation*}
  \|u\|_{L^p(M; E)} \ede\,
  \begin{cases}
    \ \left(\int_{M} \|u(x)\|_E^p \dvol_g(x)\right)^{1/p}\,, &
    \quad \mbox{if } \ p < \infty \\
    \ \, \esssup_{x \in M}\, \|u(x)\|_E  \,, &
    \quad \mbox{if } \  p = \infty\,.
  \end{cases}
\end{equation*}
Identifying two sections of $E$ that are equal on $M$ except on
a zero measure set, we then obtain the usual Lebesgue spaces
\begin{equation*}
  L^p(M; E) \ede \{ u \colon M \to E \mid \|u\|_{L^p(M; E)}
  < + \infty \}/\ker (\|\cdot \|_{L^p(M; E)})\,.
\end{equation*}
Let $(f, g) = (f, g)_{L^2(M; E)}$ denote the scalar (inner) product on $L^2(M; E)$.
It is linear in the first variable and conjugate linear in the second one.

\begin{definition}\label{def.Sobolev}
Let $E \to M$ be a finite dimensional, Hermitian vector bundle with metric preserving connection
$\nabla = \nabla^E$. Let $k \in {\mathbb Z}_+$ and $p\in [1,\infty ]$. Then
\begin{equation*}
  W^{k,p}(M; E) \ede \{u \in L^p(M; E) \mid \nabla^ju \in
  L^p(M; E \otimes T^{*\otimes j} M) \,,\, \mbox{ for }\, 1 \le j \le
k \}
\end{equation*}
is the \emph{order $k$, $L^p$-type covariant Sobolev space} of
sections of $E$ with the norm
\begin{equation*}
\|u\|_{W^{k,p}(M; E)} \ede \ell^p\mbox{--norm of } \{
\|\nabla^j (u)\|_{L^p(M; E \otimes T^{* \otimes j} M)}\,,\ 0 \le j \le k\} \,.
\end{equation*}
We shall write also $H^k(M; E) \define W^{k, 2}(M; E)$.
\end{definition}

The adjective ``covariant'' in ``covariant Sobolev spaces'' (a term coined in \cite{KohrNistor1})
serves to remind us that these spaces depend on the choice of the covariant derivative $\nabla^E$
on the coefficient bundle $E$. In \cite{KohrNistor1}, the covariant Sobolev spaces
are also called $\nabla$-Sobolev spaces. Since we use only canonical connections and
work on manifolds with bounded geometry, these spaces
coincide with the usual Sobolev spaces, so we shall drop the connection $\nabla$ or the adjective
``covariant'' when talking about these spaces.

\subsubsection{Negative order covariant Sobolev spaces}
We proceed as usual and we define the negative order (covariant) Sobolev spaces for
$p \in [1, \infty)$ by
\begin{equation*}
  W^{-k,p'}(M; E^*) \ede W^{k,p}(M; E)^* \,,
\end{equation*}
where $V^*$ is the complex conjugate dual of $V$, $k\in \NN$,
and $\frac{1}{p}+\frac{1}{p'}=1$. For simplicity, in the following, we
shall identify $E^*$ with $E$ using the Hermitian metric on $E$, so
$W^{-k,p'}(M; E^*) = W^{-k,p'}(M; E)$. We set $W^{\infty, p}(M; E) \define \cap_{k \in \NN} W^{k, p}(M; E)$
(which we shall use only for $p = \infty$ in this paper). The spaces $H^s(M; E)$,
$s \in \RR$, are defined by complex interpolation.

\subsection{Differential operators}\label{sec:gd}

Let $E, F \to M$ be vector bundles, with $E$ endowed with a connection.
A \emph{differential operator acting on sections of $E$ with values sections of $F$}
is an expression of the form
\begin{equation}\label{eq.def.do}
  Pu \ede \sum_{j=0}^\mu a_j \nabla^ju\,,
\end{equation}
with $a_j$ a section of $\Hom(E \otimes T^{*\otimes j} M; F)$. Let us
assume $E = F$ in what follows, for simplicity. Then $\Hom(E \otimes T^{*\otimes j} M; E)
\simeq \End(E) \otimes TM^{\otimes j}$, which simplifies the notation, so, in
the following, we shall identify these spaces.
Let $\rho \in \ZZ_+ \cup \{\infty\}$. We shall say that the operator 
$P$ of Equation~\eqref{eq.def.do} \emph{has coefficients in} $W^{\rho, \infty}$ if $a_j \in
W^{\rho, \infty}(M; \End(E) \otimes TM^{\otimes j})$ for all $0 \leq j\leq \mu$.
A differential operator $P = \sum_{j=0}^\mu a_j \nabla^j$
with coefficients in $W^{\rho, \infty}$ defines a \emph{continuous map}
\begin{align*}
  P \seq \sum_{j=0}^\mu a_j \nabla^j \colon W^{s+\mu, p}(M; E) \, \to \,
  W^{s, p}(M; E), \quad 0 \le s \le \rho\ \mbox{ and } \
  p \in [1, \infty] \,.
\end{align*}
See \cite{AGN1, GN17, KohrNistor1} for details and further results. It will be
convenient to use the following notation for our sets of differential operators.

\begin{definition}\label{def.Dmr}
Let $\mu \in \ZZ_+$, $\rho \in \ZZ_+ \cup \{\infty\}$, and let
$E, F\to M$ be Hermitian vector bundles with metric preserving connections.
Then $\maD^{\mu, \rho}(M;E, F)$ will denote the set of
differential operators of order $\le \mu$ with coefficients in $W^{\rho,\infty}$
acting from sections of $E$ to sections of $F$. We let $\maD^{\mu, \rho}(M;E)
\define \maD^{\mu, \rho}(M;E, E)$.
\end{definition}

In \cite{KohrNistor1}, the differential operators introduced in this section were also
called ``$\nabla$-differential operators.'' Locally, there is no difference between
the $\nabla$-differential operators and the ``usual'' differential operators. In
particular, the principal symbol is defined in the same way.

\section{Manifolds with bounded geometry}
\label{sec3}

We now recall some basic material on manifolds with
bounded geometry following \cite{AGN1}, to which we refer for more
details and references. Once we will have recalled the definition of
manifolds with bounded geometry, we will start assuming that our given
Riemannian manifold $(M, g)$ is of bounded geometry.

\subsection{Definition of bounded geometry} We now recall the definition
of a ``manifold with bounded geometry.''
Let $M$ be a Riemannian manifold with metric $g$,
as usual. We let $\dist_g(x, y)$ denote the distance between two points $x,y\in M$.

\begin{definition}\label{def_ttly_bdd_curv}
A Hermitian vector bundle $E \to M$ with metric preserving connection is said to
have \emph{totally bounded curvature} if its curvature $R^E$
and all covariant derivatives $\nabla^k R^E$, $k \in \NN$, are bounded on $M$. Also,
we shall say that $M$ has \emph{totally bounded curvature}
if the tangent bundle $TM\to M$ has totally bounded curvature.
\end{definition}

We shall use the following standard notation.

\begin{notation}\label{not.standard}
  Let us assume $M$ to be complete in its metric $g$, for simplicity.
  \begin{enumerate}[(i)]
    \item $\exp_x^M\colon T_xM \to M$ will be the geodesic exponential map at $x$.

    \item $B_x^M(r) \define \{y \in M \mid \dist_g(x, y) < r\}$,
    the ball of radius $r$ in $M$ centered at $x$.

    \item $\rinj(x)\! \define\!\! \sup\{ r \mid \exp_x^M \colon B_r^{T_x M}(0) \to M
    \text{ is a diffeomorphism onto its image} \}$.
 
     \item $\rinj(M) \define \inf_{x \in M} \rinj(x)$, the \emph{injectivity radius} of $M$.
  \end{enumerate}
\end{notation}

The concept of manifolds with bounded geometry (see, for instance, \cite{Aubin, AmannFunctSp, 
Browder61, kordyukovLp2, ShubinAsterisque, Skrz, SkAtomic, TriebelBG}) is classical and fundamental 
for this paper.

\begin{definition}\label{def_bdd_geo_no_bdy}
Let $(M,g)$ be a boundaryless Riemannian manifold. Then $M$ has \textit{bounded geometry} if
it has totally bounded curvature and $\rinj(M)>0$ (see \ref{not.standard}).
\end{definition}

We shall need the following remark.

\begin{remark}\label{rem.same.do}
  Our definition of differential operators coincides with the one in \cite{kordyukovLp1, kordyukovLp2,
  ShubinAsterisque} for manifolds with bounded geometry since in any
  normal coordinate chart defined on a ball of radius $r < \rinj(M)$, the coefficient
  of our differential operator can be expressed in terms of the Christoffel coefficients
  and of the bundle morphisms $a_j$ of Equation~\ref{eq.def.do}
  and hence they and their derivatives are uniformly bounded. This will be discussed
  in greater detail in our forthcoming paper \cite{KohrGrosseNistor2}.
\end{remark}

Of course, a smooth, compact manifold (without boundary) will have bounded
geometry. Another useful class of manifolds with bounded geometry is the class of manifolds
with cylindrical ends, see Figure \ref{fig1} for an example of such a
manifold.

\begin{figure}
\centering 
  \includegraphics[width=0.5\textwidth]{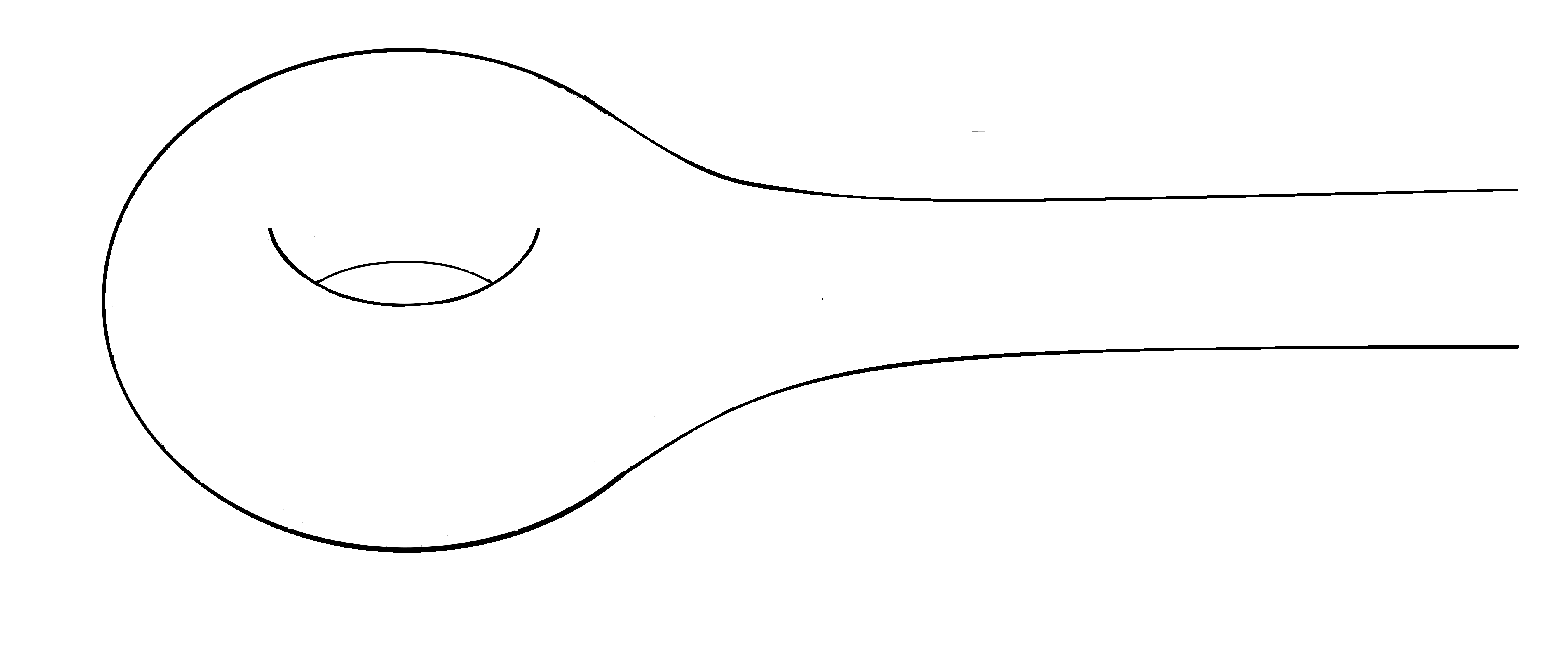}
  \caption{A manifold with cylindrical ends}\label{fig1}
\end{figure}

Some of the results that hold for these manifolds (compact or with
cylindrical ends) hold also for manifolds with bounded geometry, most notably
the ones related regularity, which are exactly the results that we will need next.
Before discussing regularity, we need to introduce, however, an alternative
definition of Sobolev spaces on manifolds with bounded geometry.

\subsection{Sobolev spaces via coverings and partitions of unity}
\label{ssec.coverings}

\emph{From now on, we will assume that $M$ is a manifold with bounded geometry and that the
vector bundles $E, F\to M$ have totally bounded curvature.} With these assumptions, we now
recall an alternative definition of the Sobolev spaces on $M$ with coefficients in $E$,
see \cite{GrosseSchneider, TriebelBG}. See also \cite{sobolev}. 

Let $\dimM$ denote the dimension of $M$.
We will write $B_r^{\dimM}(x)$ instead of $B_r^{\RR^{\dimM}}(x)$, for simplicity.
Identifying the space $T_xM$ with $\RR^{\dimM}$ via some choice of an orthonormal basis,
we obtain, for $r < \rinj(M)$, the \emph{geodesic normal coordinates} diffeomorphisms
\begin{equation}\label{eq.def.kx}
  \kappa_x \ede \exp^{M}_x \colon B^{\dimM}_{r}(0) \to B^{M}_r(x) \subset M\,.
\end{equation}

\begin{definition}\label{FC}
Let $M$ be a manifold with bounded geometry and let $0 < r\leq \frac{1}{2} \rinj (M)$.
An \emph{$r$-covering subset of $M$} is a subset $\{x_\gamma\}_{\gamma \in \NN} \subset M$
that satisfies the following conditions:
\begin{enumerate}[(i)]
  \item $M \subset \cup_{\gamma \in \NN} B_r^M(x_\gamma)$.
  \item For any $R>0$, there exists $N_R \in \NN$ with the property that, for any
  $y \in M$, the set $\{\gamma \in \NN \mid \dist_g(x_\gamma, y) <
  R\}$ contains at most $N_R$ elements.
\end{enumerate}
\end{definition}

Note that it is always possible to construct an $r$-covering set
$\{x_\gamma\}$ and the corresponding covering $\{B_r^M(x_\gamma) \}$
of $M$ is strongly locally finite (see, for instance,
\cite{sobolev, GrosseSchneider} and the references therein).
For further purposes, we need the following class of partitions of unity defined using
$r$-covering sets.

\begin{definition} \label{def_part}
A partition of unity $\{\phi_\gamma\}_{\gamma \in \NN}$ of $M$ is
called \emph{an $r$-uniform partition of unity associated to the
$r$-covering set $\{x_\gamma\} \subset M$} (Definition~\ref{FC})
if, for each $\rho  \in \ZZ_{+}$, we have $\sup_{\gamma}
\|\phi_\gamma\|_{W^{\rho ,\infty}(M)} < \infty$ and,
for each $\gamma \in \NN$, the support of $\phi_\gamma$ is contained
in $B_r^M(x_\gamma)$.
\end{definition}

In order to deal with the coefficient vector bundles $E, F \to M$,
we will also need to locally trivialize these bundles.

\begin{definition}\label{def_sync}
  Let $M$ be a Riemannian manifold with bounded geometry, $x \in M$, $r < \rinj(M)$,
  and $E \to M$ be a Hermitian vector bundle with metric preserving
  connection. Then the \emph{synchronous trivialization} of $E$ above
  $B_x^M(r)$ is the vector bundle isomorphism
  $ \xi_x \colon E|_{B_r^M(x)} \to B_r^M(x) \times E_{x}$
  induced by parallel transport along the geodesic rays emanating from $x$.
  By abuse of notation, we shall also denote by
  \begin{equation*}
    \xi_x \colon E|_{B_r^M(x)} \to B_r^{\dimM}(0) \times \CC^\tau
  \end{equation*}
  the resulting isomorphism obtained using also the geodesic coordinates diffeomorphism
  $\kappa_x \colon B_r^{\dimM}(0)\to B_r^M(x)$ (of Equation~\eqref{eq.def.kx}) and some choice
  of an isomorphism $E_x \simeq \CC^\tau$.
\end{definition}

We shall need a definition of Sobolev spaces using partitions of unity
\cite{GrosseSchneider} and a few standard results. The
following proposition is a direct consequence of Theorems 14 and
26 in \cite{GrosseSchneider}.

\begin{proposition}\label{prop.part.unity}
Let $M$ be a Riemannian manifold with bounded geometry. Let $E \to M$ be a vector
bundle over $M$ with totally bounded curvature. Let $\{\phi_\gamma\}$ be an
$r$-uniform partition of unity associated
to an $r$-covering set $\{x_\gamma\} \subset M$ (Definition~\ref{def_part})
and let $\kappa_\gamma \define \kappa_{x_\gamma} \colon B_r^{\dimM}(0)\to B_r^M(x_\gamma)$ and
$\xi_\gamma \define \xi_{x_\gamma} \colon E|_{B_r^M(x)} \to B_r^{\dimM}(0) \times \CC^\tau$
be as in Definition~\ref{def_sync}. Let $1 < p < \infty$ and $s \in \RR$. Then
  \begin{equation*}
  ||| u |||_{s,p}^{p} \define \sum_{\gamma}
  \|\xi_\gamma(\phi_\gamma u)\|_{W^{s,p}(B_r^{\dimM}(0); \CC^\tau)}^{p}
 \end{equation*}
defines a norm equivalent to the norm on $W^{s,p}(M;E)$ introduced in
Definition~\ref{def.Sobolev} using connections.
\end{proposition}

Let $\| \cdot \|_{\ell^p}$ be the norm on $\ell^p(\NN)$.
That is, if $a = (a_n)_{n \in \NN}$ is a sequence of complex numbers, then
\begin{equation*}
  \| a \|_{\ell^p} \seq \| (a_n)_{n \in \NN} \|_{\ell^p} \ede
  \begin{cases}
    \ \big ( \sum_{n \in \NN} |a_n|^p \big) ^{1/p} & \mbox{ if } p < \infty\\
    \ \sup_{n \in \NN} |a_n| & \mbox{ if } p = \infty\,.
  \end{cases}
\end{equation*}
Then $\|a\|_{p}$ is \emph{decreasing} in $p \in [1, \infty]$. 
It will be convenient to also write
\begin{equation}\label{eq.two-norms}
  ||| u |||_{s,p} \seq 
  \big \| \big ( \|\xi_\gamma(\phi_\gamma u)\|_{W^{s,p}} \big)_{\gamma \in \NN}
  \big \|_{\ell^p}\,.
 \end{equation}
with $\|\xi_\gamma(\phi_\gamma u)\|_{W^{s,p}} = 
\|\xi_\gamma(\phi_\gamma u)\|_{W^{s,p}(B_r^{\dimM}(0); \CC^\tau)}$.

The Rellich-Kondrachov compact embedding theorem
\cite[Theorem~6.3]{AdamsBook} (see also \cite{AdamsFournier, AmannSobolev, Kondrachov-comp, Taylor3})
gives then the following extension to manifolds with bounded geometry of the
results of Rellich-Kondrachov and Lions. Related results were obtained in \cite{Skrz}
(but using weighted spaces) and in \cite[Appendix A]{Grosse12} for $p = 2$.

\begin{proposition}\label{prop.l.reg}
  Let $p, p', q \in (1, \infty)$, $s, s'\in \ZZ_+$, $t \in \ZZ$, 
  satisfy the Rellich-Kondrachov conditions:
  \begin{equation*}
    t < s < s' \quad \mbox{ and }\quad {\max\{t, 0\}} - n/q \le s - n/p \le s' - n/p' \,.
  \end{equation*}

  Let also $M$ be a manifold with bounded geometry and $E \to M$ be a hermitian vector
  bundle with totally bounded curvature.
  \begin{enumerate}[(i)]
    \item If also $p' \le p$, then we have $W^{s', p'}(M; E) \subset W^{s, p}(M; E)$ 
    (continuous embedding).

    \item Assume furthermore that $\max\{ p', q \} \le p$ and $s - n/p < s' - n/p'$. Then,
    for every $\epsilon > 0$, there exists $C_\epsilon > 0$ such that,
    for all $u \in W^{s', p'}(M; E)$, we have
    $$ \|u\|_{W^{s, p}(M; E)} \le \epsilon \|u\|_{W^{s', p'}(M; E)}
    + C_\epsilon \|u\|_{W^{t, q}(M; E)}\,.$$
  \end{enumerate}
\end{proposition}

\begin{proof}
Both statements follows from the analogous statements over an open ball
$M_0 \subset \RR^\dimM$ together with Proposition~\ref{prop.part.unity}.

Let us see how this is done. We first review the case when $M$ is 
replaced with an open ball $M_0$ in a euclidean spaces. We may assume then that 
$E$ is replaced with a trivial bundle $E_0 \to M_0$. The first
statement for $(M_0, E_0)$ is then nothing but the classical Sobolev embedding,  
\cite[Theorem~5.4]{AdamsBook}. 

Let us recall next a proof of the second statement for
$(M_0, E_0)$ (recall that $M_0$ is an open ball). 
Let us assume that the statement is not true. Let
us fix $\epsilon > 0$. Then, there exists a sequence $u_n \in W^{s', p'}(M_0, E_0)$
such that
$$ \|u_n\|_{W^{s, p}(M_0, E_0)} > \epsilon \|u_n\|_{W^{s', p'}(M_0, E_0)} + n
\|u_n\|_{W^{t, q}(M_0, E_0)}\,.$$
We can assume that $\|u_n\|_{W^{s', p'}(M_0, E_0)} = 1$. In view of the Rellich-Kondrachov
compactness theorem, the natural inclusions
$W^{s', p'}(M_0, E_0) \to W^{s, p}(M_0, E_0)$
is compact, and hence (eventually after replacing the sequence $(u_n)_{n \in \NN}$
with a subsequence), we may also assume that the sequence
$u_n$ converges in norm to $w \in W^{s, p}(M_0, E_0)$. Passing to the limit in the
relation \[ \|u_n\|_{W^{s, p}(M_0, E_0)} > \epsilon \|u_n\|_{W^{s', p'}(M_0, E_0)} = \epsilon,\]
we obtain $ \|w\|_{W^{s, p}(M_0, E_0)} \ge \epsilon > 0$. Hence $w \neq 0$.
Let us notice that the assumptions also give that $W^{s, p}(M_0, E_0) \to W^{t, q}(M_0, E_0)$
is continuous. Indeed, if $t \ge 0$, this follows directly from the Rellich-Kondrachov embedding.
If $t < 0$, then it follows from the continuity of $W^{s, p}(M_0, E_0) \to W^{0, q}(M_0, E_0)$
and of $W^{0, q}(M_0, E_0) \to W^{t, q}(M_0, E_0)$.
The continuity of $W^{s, p}(M_0, E_0) \to W^{t, q}(M_0, E_0)$ gives that $u_n$ converges to
$w$ also in $W^{t, q}(M_0, E_0)$. Let $R \define \sup_{n \in \NN} \| u_n \|_{W^{s, p}(M_0, E_0)}$. We have
$R < \infty$ since $u_n$ converges in $W^{s, p}(M_0, E_0)$. Then, passing to the limit in the relation
$$R/n \ge \|u_n\|_{W^{s, p}(M_0, E_0)}/n >  \|u_n\|_{W^{t, q}(M_0, E_0)}\,,$$ we
obtain $\|w\|_{W^{t, q}(M_0, E_0)} = 0$. This is in contradiction with the previously
proved relation $w\neq 0$. Thus the
statement (ii) is true if $M = M_0$, an open ball in $\RR^\dimM$.

Let us next prove (i) in general. Let $u \in W^{s',p'}(M; E)$. We shall write $A \lesssim B$
if there is a bound $C \ge 0$ independent of the functions involved and $\epsilon >0$ 
such that $A \le CB$. We shall use below also Equation~\eqref{eq.two-norms}. 
Then $\|\xi_\gamma(\phi_\gamma u)\|_{W^{s,p}} \lesssim \|\xi_\gamma(\phi_\gamma u)\|_{W^{s',p'}}$
by the case of a ball in $\RR^\dimM$ (the Rellich-Kondrachov inequality), and hence 
\begin{align*}
  \|u\|_{W^{s, p}(M; E)}\ & \lesssim ||| u |||_{s,p}  \ede 
  \big \| \big ( \|\xi_\gamma(\phi_\gamma u)\|_{W^{s,p}} \big)_{\gamma \in \NN}
  \big \|_{\ell^p}\\ 
  & \lesssim \big \| \big ( \|\xi_\gamma(\phi_\gamma u)\|_{W^{s',p'}} \big)_{\gamma \in \NN}
  \big \|_{\ell^p}\\ 
  & \leq \big \| \big ( \|\xi_\gamma(\phi_\gamma u)\|_{W^{s',p'}} \big)_{\gamma \in \NN}
  \big \|_{\ell^{p'}}  \ =:\ ||| u |||_{s',p'}
   \, \lesssim\, \|u\|_{W^{s', p'}(M; E)}\,,
\end{align*}
where we have used also the fact that $\|\cdot \|_{\ell^p}$ is decreasing in $p$. This proves (i).

Let us now prove (ii) in the general case. We have successively:
\begin{align*}
  \|u\|_{W^{s, p}(M; E)}  & \lesssim \big \| \big ( \|\xi_\gamma(\phi_\gamma u)\|_{W^{s,p}} \big)_{\gamma \in \NN}
  \big \|_{\ell^p}\\ 
  & \leq \big \| \big ( \epsilon \|\xi_\gamma(\phi_\gamma u)\|_{W^{s',p'}} 
  + C_\epsilon \|\xi_\gamma(\phi_\gamma u)\|_{W^{t,q}} \big)_{\gamma \in \NN}
  \big \|_{\ell^p}\\ 
  & \leq \epsilon  \big \| \big ( \|\xi_\gamma(\phi_\gamma u)\|_{W^{s',p'}} \big)_{\gamma \in \NN}
  \big \|_{\ell^p}
  + C_\epsilon \big \| \big ( \|\xi_\gamma(\phi_\gamma u)\|_{W^{t,q}} \big )_{\gamma \in \NN}
  \big \|_{\ell^p}\\ 
  & \leq \epsilon  \big \| \big ( \|\xi_\gamma(\phi_\gamma u)\|_{W^{s',p'}} \big)_{\gamma \in \NN}
  \big \|_{\ell^{p'}}
  + C_\epsilon \big \| \big ( \|\xi_\gamma(\phi_\gamma u)\|_{W^{t,q}} \big )_{\gamma \in \NN}
  \big \|_{\ell^q} \\
  &  \lesssim \epsilon \|u\|_{W^{s', p'}(M; E)} + C_\epsilon \|u\|_{W^{t, q}(M; E)} \,,
\end{align*}
where we have used again the decreasing monotony of the $\ell^p$-norm and $p\ge p'$ and $p \ge q$.
This means $\|u\|_{W^{s, p}(M; E)} \le C \epsilon \|u\|_{W^{s', p'}(M; E)} + C C_\epsilon \|u\|_{W^{t, q}(M; E)}$.
To obtain the desired result, one needs only to rescale $\epsilon$ by replacing it with $\epsilon/C$.
\end{proof}

Of course, if $p = p' = q = 2$, then one can prove the above result using also
interpolation inequalities. See also \cite{AmannFunctSp, sobolev, Aubin, HebeyBook, HebeyRobert,
TriebelBG} for related results, in particular, for the use of the partitions of
unity. See \cite{BrezisBook, LionsMagenes1, Taylor1, TriebelBook, TriebelEll} for the general
results on Sobolev spaces not proved above.

\section{Regularity estimates for second order strongly elliptic operators}
\label{sec4}

We shall need the regularity result for uniformly strongly elliptic operators on
manifolds bounded geometry $M$ stated in the Introduction (and which is a slight 
generalization of a result in \cite{GN17}). Note, however, that in \cite{GN17}, the 
authors have used operators in \emph{divergence form} (unlike in this paper), so
the regularity requirements on the coefficients in that and this paper are
(slightly) different.

\subsection{Regularity conditions}
Recall the notation $\maD^{\mu, \rho}(M;E, F)$ for our order $\mu$ differential 
operators with coefficients in $W^{\rho,\infty}$ (Definition \ref{def.Dmr}).

\begin{definition} \label{def.strongly.elliptic}
  Let $P \in \maD^{2, 0}(M; E)$. We shall say that $P$ is
  \emph{uniformly strongly elliptic} if there exists $\gamma_P > 0$ such that, for all
  $x \in M$ and all $\eta \in T_x^* M$ and $\zeta \in E_x$, we
  have that its principal symbol $\sigma_2(P) \colon T^*M \to \End(E)$ satisfies
  the inequality
  \begin{equation}\label{eq.def.se}
    \operatorname{Re} (\sigma_2(P)(\eta) \zeta, \zeta)_{E_x} 
    \ge \gamma_P \|\eta\|^2 \|\zeta\|^2\,.
  \end{equation}
\end{definition}

Of course, the defining relation \eqref{eq.def.se} is equivalent to 
\begin{equation}
  \sigma_2(P)(\eta) + \sigma_2(P)(\eta)^* \ge 2\gamma_P \|\eta\|^2
\end{equation}
(with the inequality in operator sense).
Let us now prove our regularity result stated in the Introduction.

\begin{proof}[Proof of Theorem~\ref{thm.reg.D}]
  A differential operator with coefficients in $W^{\rho + 3, \infty}$ can
  be written as a differential operator in divergence form with coefficients
  in $W^{\rho + 2, \infty}$. (Indeed, in local coordinates, an operator 
  in divergence form is a linear combination of operators of the form $\pa_i a \pa_j$.
  The claim is obtained by writing $\pa_i a \pa_j = a \pa_i \pa_j + \pa_i(a) \pa_j$,
  with the right hand side consisting of operators written in the standard form.
  The loss of one order of regularity is due to the appearance of $\pa_i(a)$.)
  For $t = 1$ and operators in divergence form with
  coefficients in $W^{\rho + 2, \infty}$, this result was proved in
  \cite[Theorem 5.12]{GN17}. This completes our proof if $t = 1$.
  Proposition~\ref{prop.l.reg} for  $\epsilon c_{s,1} < 1$, $p = p' = q = 2$,
  and $s' = s +2$ then allows us to extend to the case $t < 1$.
  The result for $t > 1$ follows from the result for $t = 1$. The proof is
  thus complete.
\end{proof}

A more general result than this theorem will be provided in \cite{GKN-23}.
We immediately obtain the following result (which eliminates the restriction
$s \le \rho$ of the above theorem).

\begin{corollary}\label{cor.regularity}
  We use the notation and assumptions of Theorem~\ref{thm.reg.D}, except
  that we assume that $P$ has coefficients in $W^{\infty, \infty}$, but allow $s \in \NN$.
  Then there exists $c_{s,t} > 0$ such that, if $u \in H^{t}(M; E)$ and $Pu \in H^{s}(M; F)$, 
  it follows that $u \in H^{s + 2}(M; E)$ and
  \begin{equation*}
    \| u \|_{H^{s + 2}(M; E)} \leq c_{s,t}
    \left( \|P u\|_{H^{s}(M; F)} + \| u \|_{H^{t}(M; E)} \right).
  \end{equation*}
\end{corollary}

The above result will be used (mostly) for $s = t = 0$, which explains why we needed
to include the case $t=0$ in our regularity result, Theorem \ref{thm.reg.D}.
As already alluded to above, these regularity results are closely related
to the results in \cite{GN17}, but also to the ones in \cite{kordyukovLp1, kordyukovLp2, 
SkrzMapping}.

\begin{remark}
As noticed already in Remark~\ref{rem.same.do}, a differential operator in our
sense (i.e. a $\nabla$-differential operator) is also a uniform differential
operator in the sense of \cite{kordyukovLp1, kordyukovLp2}. Thus, if $P$ is
\emph{scalar}, then Corollary~\ref{cor.regularity} was stated also in
\cite{kordyukovLp1, kordyukovLp2}, where it was also noticed
that it can be proved as in the classical case of compact manifolds. 
See also Proposition 4 in \cite{Skrz}. 
Clearly, these classical proofs are also valid for uniformly strongly elliptic 
\emph{systems of operators.} On a different note, the restriction $t \ge 1$ in \cite{GN17} 
was due to the fact that, in that paper, the authors worked on
manifolds with boundary. In principle, the method of that paper could also 
give the result for $t \in \ZZ$, when there is no boundary.
\end{remark}

\subsection{Consequences of regularity}
We continue to assume that $M$ is a manifold with bounded geometry and that
$E,F \to M$ are two Hermitian vector bundles with totally bounded curvature.
We shall need the following standard consequences of our regularity results.
For simplicity, we assume that all our differential operators $P$ have
smooth coefficients. Let then
\begin{equation}\label{eq.def.d.inf}
  \maD^{\mu,\infty}(M; E, F) \ede \bigcap _{\rho  \in \NN} \maD^{\mu, \rho}(M; E, F).
\end{equation}
(Compare with Definition \ref{def.Dmr}, where the spaces $\maD^{\mu, \rho}(M; E, F)$
of differential operators were introduced.)
Corollary~\ref{cor.regularity} yields the following results.

\begin{corollary} \label{cor.closed}
  Let $P \in \maD^{2,\infty}(M; E, F)$ be a uniformly strongly elliptic
  differential operator  (see Equation \ref{eq.def.d.inf} for notation). We regard $P$ as an \emph{unbounded} operator
  on $H^s(M; E)$ with domain  $H^{s+2}(M; E) \subset H^s(M; E)$ and with values
  in $H^s(M; F)$. Then $P$ is closed.
\end{corollary}

\begin{proof} Let $\eta_n \in H^{s+2}(M; E)$, $\eta\in H^{s}(M; E)$, and
  $\zeta \in H^{s}(M; F)$ be such that
  $\eta_n \to \eta$ and $P \eta_n \to \zeta$, the first convergence being in $H^{s}(M; E)$
  and the second one being in $H^{s}(M; F)$.
  Then, since $P$ is continuous in distribution sense, we have $P \eta_n \to P \eta$ in distribution
  sense, and   hence $P\eta = \zeta$ in distribution sense.
  Theorem~\ref{thm.reg.D} for $u = \eta$ and
  $t = s$ gives that $\eta \in H^{s+2}(M; E)$. Hence the graph of $P$ is closed.
\end{proof}

\section{Application: \tps{L^2}{L2}-unique continuation and the deformation operator}
\label{sec5}

We continue to assume that $M$ is a manifold with bounded geometry and that
$E \to M$ is a Hermitian vector bundle with totally bounded curvature.

\subsection{An invertibility result}

The next result states that if $T \in \maD^{2, \infty}(M; E)$ is a uniformly strongly elliptic symmetric
operator on $L^2(M;E)$, on which we consider the inner product $(\cdot ,\cdot )_{L^2(M;E)}$,
then $T$ is essentially self-adjoint, that is, its closure is a self-adjoint operator
on $L^2(M;E)$. Its proof is standard.

\begin{lemma}
\label{elliptic-self-adj}
Let $M$ be a manifold with bounded geometry and $E \to M$ be a Hermitian vector bundle
with totally bounded curvature. Let $T\in \maD^{2, \infty}(M; E)$ be a uniformly strongly elliptic
differential operator and consider the setting of Corollary \ref{cor.closed} for $s=0$. 
If $T\colon H^2(M;E)\to L^2(M;E)$ is symmetric $($that is, $(Tf,g)=(f,Tg)$, for all 
$f,g\in H^2(M;E)${$)$}, then $T$ is self-adjoint.
\end{lemma}

\begin{proof}
  Let $T^\dagger \colon \CIc(M; E)^* \to \CIc(M; E)^*$ be the adjoint of $T$ in distribution 
  sense (with $V^*:=\overline{V}'$ the conjugate linear dual of $V$). Our
  assumption is that $T^\dagger = T$ as \emph{differential operators acting on
  $\CIc(M; E)$}, and hence also as \emph{differential operators acting on
  distributions}, since $T$ and $T^\dagger$ are continuous on distributions and
  $\CIc(M; E)$ is dense in the space $\CIc(M; E)^*$ of distributions. Let $T^*$ be 
  the adjoint of $T$ in $L^2$-sense. (The pairing between $V$ and $V^*$ is
  sesquilinear.) Since we know that $T$ is closed, by Corollary \ref{cor.closed},
  it is enough to show that $T^* \pm \imath$ are both injective (where $\imath\in \mathbb C$
  is such that $\imath^2=-1$). Let $\phi$ in
  the domain of $T^*$ be such that $(T^* + \imath) \phi = 0$. Since $\phi \in L^2(M; E)$,
  it is a distribution, and $T^\dagger \phi$ is defined. By taking inner products with
  test functions $\psi \in \CIc(M; E)$ (which are in the domain of $T$), we see that
  $$ (T + \imath) \phi = (T^\dagger + \imath) \phi = (T^* + \imath) \phi = 0\,,$$
  in distribution sense. Since $T$ is elliptic of order $2 > 0$, we obtain that
  $T + \imath$ is also elliptic of order $2$. The regularity theorem, Theorem~\ref{thm.reg.D},
  gives that $\phi \in H^2(M; E)$. Hence $\phi$ is in the domain of $T$. But then the usual
  $L^2$-estimate
  \begin{equation*}
    0 = \|(T + \imath ) \phi \|_{L^2}^2 \seq ((T + \imath ) \phi, (T + \imath ) \phi) \seq
    \|T\phi \|_{L^2}^2 + \|\phi\|_{L^2}^2
  \end{equation*}
  gives that $\phi = 0$.
\end{proof}

Recall the definition of the $L^2$-unique continuation property from Definition~\ref{def-L2-cont}.
The most basic example is the following.

\begin{example}
  Let $M_0$ be a connected Riemannian manifold. The exterior derivative
  $d \colon L^2(M_0) \to \CIc(M_0, T^*M_0)'$ has the $L^2$-unique continuation property. Indeed,
  if $d u = 0$, then $u$ is locally constant, and hence constant, since $M_0$
  was assumed to be connected. If, furthermore, $u$ vanishes on some non-empty
  open subset of $M_0$, then this constant must be zero.
\end{example}

The following result is a key technical application of $L^2$-unique continuation
property that allows us to establish the invertibility of our operators
in applications to the Stokes system \cite{KohrNistorWendland1}. It motivates the 
consideration of the $L^2$-unique continuation property. Let $X$ and $Y$ be
Banach spaces. Recall that a (possibly 
unbounded) closed operator $T : \maD(T) \subset X \to Y$ is \emph{Fredholm} if
$T^{-1}(0)$ and $Y/TX$ are finite dimensional. Then its \emph{index} is $\dim T^{-1}(0)
- \dim (Y/TX)$.

\begin{proposition}  \label{Fredholm-invertible}
  Let $M$ be a manifold with bounded geometry and $E \to M$ be a Hermitian vector bundle
  with totally bounded curvature. Let $L\in \maD^{2,\infty}(M; E)$ be a differential operator 
  (so of order $2$) and $V \in W^{\infty,\infty}(M;\End (E))$ (so of order zero) with the 
  following properties:
  \begin{enumerate}[(i)]
    \item \label{item.FI.i} $V, L \ge 0$, more precisely, for all $v\in \CIc(M;E)$,
    \begin{equation*}  
      ( Lv,v) \geq 0 \ \mbox{ and } \ ( Vv,v) \geq 0\,.
    \end{equation*}

    \item \label{item.FI.ii} $L$ is uniformly strongly elliptic
    and has the $L^2$-unique continuation property (Definition~\ref{def-L2-cont}).
     \item \label{item.FI.iii} There exists a non-empty open set ${\maO}\subset M$
    such that, if $v\in L^2(M;E)$ and  $( V v,v) =0$, then $v|_{{\maO}} \seq 0.$
    \item \label{item.FI.iv} $L+V \colon H^2(M;E)\to L^2(M;E)$ is a Fredholm operator.
  \end{enumerate}
Then $L + V \colon H^2(M;E)\to L^2(M;E)$ is invertible.
\end{proposition}

\begin{proof}
Since $V$ has order smaller than the order of $L$ and $L$ is uniformly strongly
elliptic, by assumption, it follows that $L+V$ is also uniformly strongly elliptic.
We then notice that, since $L$ and $L + V$ are both uniformly strongly elliptic, 
both have domain $H^2(M;E)$ and are closed as unbounded operators on $L^2(M; E)$
by Corollary \ref{cor.closed}. The positivity of both $L$
and $V$ implies also the positivity of $L + V$. (We obtain positivity on the whole
domain since $\CIc(M; E)$ is dense in $H^2(M; E)$ in the graph norm.) In turn, positivity
implies that $L$ and $L+V$ are symmetric (see, for instance, Proposition 2.12 in \cite{ConwayBook}). 
Lemma~\ref{elliptic-self-adj} then gives that both $L$ and
$L+V$ are self-adjoint. Since $L+V$ is Fredholm and self-adjoint, its index will be 0.

Let us assume, by contradiction, that $L + V$ is \emph{not} invertible. Since it is
Fredholm of index zero, it must have a non-zero kernel. Thus, there exists $u \in H^2(M; E)$,
$u\neq 0$, such that $(L+V)u = 0$. Then
\[ (Lu,u) + (Vu,u) =0\,.\]
The relations $(Lu,u) \geq 0$, $(Vu,u) \geq 0$,
and the self-adjointness of $L$, as unbounded operator on $L^2(M;E)$, yield
\[\|L^{\frac{1}{2}}u\|^2=( Lu,u) =0 \ \mbox{ and } \ ( V u,u) =0\,.\]
Hence, we obtain $Lu=0$ on $M$.
Moreover, assumption \eqref{item.FI.iii} and the relation $(V u,u) =0$
imply that $u=0$ on the open subset ${\maO}$ of $M$. The unique continuation property satisfied by the
operator $L$ finally shows that $u = 0$ on $M$, which is a contradiction with the initial assumption that
$u \neq 0$ is in the kernel of $L + V$. The proof is complete.
\end{proof}

See \cite{Mitrea-Nistor} for a motivation and references related
to the above result. We shall need the following simple lemma.

\begin{lemma}\label{lemma-ucT*T}
  Let $M$ be a manifold with bounded geometry and $E, F \to M$ be Hermitian vector bundles with
  totally bounded curvature. Let $T \in \maD^{1,\infty}(M; E, F)$ have the
  $L^2$-unique continuation property. If $T^*T$ is uniformly strongly elliptic, then it also has the $L^2$-unique
  continuation property.
\end{lemma}

\begin{proof}
    Let $u \in L^2(M; E)$ be such that $T^*T u = 0$
    and $u = 0$ on some non-empty open set $\maO \subset M$. Since $T^*T$
    is uniformly strongly elliptic, we obtain that $u \in H^{2}(M; E)$ by elliptic regularity
    (see Theorem~\ref{thm.reg.D} or Corollary~\ref{cor.regularity}). Therefore, we have
    \begin{equation*}
      \|Tu \|^2_{L^2} \seq (Tu, Tu)  \seq (T^*Tu, u) \seq 0\,.
    \end{equation*}
    Hence $Tu = 0$. The $L^2$-unique continuation property for $T$ then gives $u = 0$.
\end{proof}

\subsection{Vanishing of Killing vector fields}
We shall need the following lemma (a stronger result can be found in
\cite[Prop. 8.1.4]{Petersen}, we include, nevertheless a simple proof, for the
benefit of the reader):

\begin{lemma}\label{lemma.Killing}
   Let $M_0$ be a connected Riemannian manifold with positive injectivity radius
  (and hence complete).
      Let $\maO \subset M_0$ be a closed subset with non-empty
    interior and with the following property:
    \begin{center} \emph{ ``If $\gamma \colon \RR \to M_0$ is a (complete)
      geodesic that is minimizing between two points $a, b \in \maO$,
      then $\gamma \subset \maO$.''}
    \end{center}
    Then $\maO = M_0$.
\end{lemma}

\begin{proof}
  Let $\maO_0 \subset \maO$ be the interior of $\maO$, which is known to be non-empty
  by hypothesis. Let us show that $\maO_0$ is also closed in $M_0$.
  For $x \in M_0$, we shall let  $B_{r}^{M_0}(x) \subset M_0$ denote the open ball with center $x$ and radius $r$
  (the set of points $y \in M_0$ at distance $<r$ to $x$). Let
  $r_0 > 0$ be the injectivity radius of $M_0$, we claim that
  \begin{equation}\label{eq.claim}
    x_0 \in \maO_0 \ \Rightarrow\ B_{r_0}^{M_0}(x_0) \subset \maO_0\,.
  \end{equation}
  Indeed, to prove the claim \eqref{eq.claim}, let $y \in B_{r_0}^{M_0}(x_0, r_0)$.
  Then there is a minimizing geodesic $\gamma  \colon \RR \to M_0$ joining $y$ to $x_0$. We can assume $\gamma(0) = x_0$.
  Then, for $t$ small enough, $\gamma(t) \in \maO_0$ and $\gamma$ is minimizing between $x_0 = \gamma(0)$
  and $\gamma(t)$. It follows by assumption that $\gamma \subset \maO$ and hence, in particular, $y \in \maO$.
  This shows that, for every $x_0 \in \maO_0$, the ball $B_{r_0}^{M_0}(x_0) \subset \maO$. Hence also
  $B_{r_0}^{M_0}(x_0) \subset \maO_0$ since $B_{r_0}^{M_0}(x_0)$ is open. We have thus proved our claim, 
  that is, Equation \eqref{eq.claim}.
  
  Let us next show that the proven claim implies that $\maO_0$
  is closed. Indeed, let $y \in \overline{\maO}_0$. Then $B^{M_0}_{r_0}(y)\cap \maO_0$ 
  is non-empty. Let $x_0
  \in B_{r_0}^{M_0}(y)\cap \maO_0$. Then $y \in B_{r_0}^{M_0}(x_0)\subset \maO_0$. Thus $\maO_0$ is
  closed, as stated.
  
  This completes the proof. Indeed, $M_0$ is connected and $\maO_0 \subset M_0$ is open, closed, and 
  non-empty, it follows that $\maO_0 = M_0$. This proves that $\maO = M_0$ as well, and concludes the proof.
\end{proof}

Let $\mathfrak{L}_X$ denote the Lie derivative in the direction of a vector field $X$. Then the
deformation tensor $\Def X$ is defined as
\begin{align}\label{Def}
	\Def X=\frac{1}{2}\mathfrak{L}_Xg
\end{align}
(see \cite[Chapter 2 \S 3]{Taylor1}). Thus, $\Def X$ is a symmetric tensor field of type $(0,2)$,
and $\Def$ will be viewed as an unbounded operator from $L^2(TM)$ to $L^2(T^*M\otimes T^*M)$.
Let also $\Def^*$ denote the adjoint of the deformation operator $\Def$. If $\Def X=0$, then
$X$ is called a \emph{Killing vector field.}

The next corollary will provide us with the $L^2$-unique continuation property for the 
deformation operator $\Def$.

\begin{corollary}
\label{cor.item.lK.b}
Let $M_0$ be a connected Riemannian manifold with positive injectivity radius.
Let $X$ be a smooth \emph{Killing} vector field on $M_0$ that vanishes on a non-empty open
subset of $M_0$. Then $X = 0$ $($i.e. it vanishes on $M_0${$)$}.
\end{corollary}

\begin{proof}
  Let $\maO\subset M_0$ be the set of points where $X = 0$,
  which contains an open subset of $M_0$, by assumption. Since $X$ is smooth and $M_0$ is complete, 
  $X$ generates a \emph{local} flow $\phi_t \colon M_0 \to M_0$, with $\phi_t(x)$ defined for
  $t \in (-\epsilon_x, \epsilon_x)$, for some $\epsilon_x > 0$. 
   Since $X$ is a Killing vector field, the local flow $\phi_t$
  consists of (partially defined) isometries. Moreover, $\phi_t = id$ on $\maO$ (since $X = 0$ there).
  We conclude that $\phi_t$ fixes any minimizing
  geodesic between two points of $\maO$. Hence any (complete) geodesic minimizing the distance between
  two points of $\maO$ is contained in $\maO$. Lemma~\ref{lemma.Killing}
  then gives that $\maO = M_0$. Hence $X = 0$ everywhere on $M$.
\end{proof}

\subsection{Applications to the deformation operator}
We continue to assume that $M$ is a manifold with bounded geometry. Let $g$ be a
Riemannian metric tensor on $M$.
We now prove our main application to the deformation operator.

\begin{theorem}
  \label{Killing-non-compact}
  Let $M$ be a manifold with bounded geometry.
  Then the following statements hold.
    \begin{enumerate}[(i)]
    \item\label{item.Knc.a}
    We have $\Def \in \maD^{1,\infty}(M; TM, T^*M \otimes T^*M)$.

    \item\label{item.Knc.b}
    The deformation Laplacian $\mathbf{L} \ede 2\operatorname{Def}^{*}\Def$ is 
    uniformly strongly elliptic.

    \item\label{item.Knc.c}
    If $M$ is connected, then $\Def$ and $\mathbf{L} \ede 2\operatorname{Def}^{*}\Def$ have
    the $L^2$-unique continuation property on $M$.
    \end{enumerate}
\end{theorem}

\begin{proof}
  \eqref{item.Knc.a} follows from formula \eqref{Def}
  and the results in \cite{KohrNistor1} (see, for instance, \cite[Proposition~5.12]{KohrNistor1}, 
  noting that manifolds with bounded geometry satisfy the (FFC) condition
  by the results in \cite{GN17}, Lemma 3.1).

  \eqref{item.Knc.b} The operator $\mathbf{L} \ede 2\operatorname{Def}^{*}\Def$ is in
  $\maD^{2,\infty}(M; \End(TM))$ by point \eqref{item.Knc.a}
  already proved and \cite[Proposition 4.5]{KohrNistor1}.
  The principal symbol of $\mathbf{L}$
  (a second order differential operator) is given by \cite[Chapter 5 \S 12, (12.20)]{Taylor1}
  \begin{equation*}
     \sigma _2(\mathbf{L})(\eta )=|\eta|^2({1 }+P_{\eta})\geq |\eta|^2\,,\ \forall\, \eta\in T_x^*M\setminus \{0\},
  \end{equation*}
  where $P_{\eta}(u):=|\eta|^{-2}( \eta,u) \eta$ is the orthogonal projection of $u$ along $\eta$ (regarded
  as a vector) and we use the notation $1$ for the identity operator.
  Hence $\mathbf{L}$ is a uniformly strongly elliptic operator, as asserted.

  \eqref{item.Knc.c} Let us assume that $M$ is a connected manifold and let $X \in L^2(M,TM)$ be such
  that $\Def X =0$ on $M$. Then $\operatorname{Def}^{*}\Def X = 0$. Since $\operatorname{Def}^{*}\Def$
  is uniformly strongly elliptic by point \eqref{item.Knc.b} already proved, elliptic regularity
  (Corollary~\ref{cor.regularity}) gives
  that $X$ is smooth. Since $M$ has positive injectivity radius, Lemma~\ref{lemma.Killing} and,
  more precisely, Corollary~\ref{cor.item.lK.b}, give
  that $X = 0$. Hence $\Def$ has the $L^2$ unique continuation property. Lemma~\ref{lemma-ucT*T}, in
  conjunction again with point \eqref{item.Knc.b} already proved, gives then that $\operatorname{Def}^{*}\Def$
  also has the $L^2$ unique continuation property.
\end{proof}

Combining with Proposition \ref{Fredholm-invertible}, we obtain the following consequence.

\begin{corollary}  \label{cor.Fredholm-invertible} 
  Let $M$ be a \emph{connected} manifold with bounded geometry 
  and $\mathbf{L} := 2 \operatorname{Def}^* \Def$, as before. Let 
  $V \in W^{\infty,\infty}(M;\End (E))$ with the following properties:
  \begin{enumerate}[$(i)$]
    \item \label{item.cFI.i} $V \ge 0$.

    \item \label{item.cFI.ii} There exists a non-empty open subset ${\maO}\subset M$
    such that, if $v\in L^2(M;E)$ and  $( V v,v) =0$, then $v|_{{\maO}} \seq 0.$

  \item \label{item.cFI.iii} $\mathbf{L} + V \colon H^2(M;E)\to L^2(M;E)$ is a Fredholm operator.
  \end{enumerate}
Then $\mathbf{L} + V \colon H^2(M;E)\to L^2(M;E)$ is invertible.
\end{corollary}

There is an analogous definition of $W^{s,p}$-unique continuation property:

\begin{definition} 
  Let $T$ be pseudodifferential operator acting from sections of $E$ to sections of $F$.
  Let us assume that $T$ extends by continuity to a map $T : W^{s, p}(M; E) \to \maD(M; F)'$. 
  We shall say that $T$ has the \textit{$W^{s, p}$-unique continuation property} if, 
  given $u \in W^{s,p}(M; E)$ that vanishes in an open subset $U \subset M_0$ and 
  satisfies $T  u \seq 0$, then $u \seq 0$.
\end{definition}

Then the same proof gives as the proof of the $L^2$-unique continuation property for
$\Def$ in Theorem \ref{Killing-non-compact}\eqref{item.Knc.c} gives the following.

\begin{corollary}
  \label{cor.Def}
  Let $M_0$ be a connected manifold with bounded geometry and $s \in \RR_+$,
  then $\Def$ has the $H^s$-unique continuation property. 
\end{corollary}

Generalizations of this kind are possible for other statements, when suitable
reformulated. For instance, the problem in Lemma \ref{lemma-ucT*T} is to specify 
which adjoint one uses. One can still use the adjoint with respect to the $L^2$-inner
product (not with respect to $H^s$), then, as a consequence, one would obtain also 
that $\mathbf{L} := \operatorname{Def}^*\Def$ also has the $H^s$-unique continuation 
property.

\def\cprime{$'$}

\end{document}